\documentclass{amsart}
\usepackage{amsmath,amsthm,amsfonts,amssymb}
\usepackage{mathrsfs}
\usepackage{latexsym}
\usepackage[latin1]{inputenc}
\usepackage{bm}
\usepackage{amssymb}
\usepackage[all]{xy}
\usepackage{enumerate}
\usepackage{stmaryrd}
\usepackage{color}
\usepackage{mathrsfs}
\usepackage{url}
\usepackage{soul}
\usepackage{ dsfont }
\usepackage{cite}
\usepackage{hyperref}

\newcommand{\mA}{{\mathcal{A}}}
\newcommand{\mF}{{\mathcal{F}}}
\newcommand{\mI}{{\mathcal{I}}}
\newcommand{\mL}{{\mathcal{L}}}
\newcommand{\mR}{{\mathcal{R}}}
\newcommand{\mW}{{\mathcal{W}}}
\newcommand{\TT}{\mathbb{T}}
\newcommand{\QQ}{{\mathbb{Q}}}

\newcommand{\OO}{{\mathbb{O}}}
\newcommand{\RR}{\mathbb{R}}

\newcommand{\spec}[1]{\operatorname{Sp}\left(#1\right)}
\newcommand{\PP}{\mathbb{P}}

\newtheorem{theorem}{Theorem}[section]
\newtheorem{lemma}[theorem]{Lemma}
\newtheorem{proposition}[theorem]{Proposition}

\newtheorem{corollary}[theorem]{Corollary}
\newtheorem{example}[theorem]{Example}

\newtheorem{question}[theorem]{Question}

\begin{document}
\title{A characterization of productive cellularity}
\author[R. M. Mezabarba]{Renan M. Mezabarba$^1$}
\thanks{$^1$ Supported by CNPq (2017/09252-3) and Capes (88882.315491/2019-01)}
\address{Centro de Ci\^encias Exatas,
Universidade Federal do Espírito Santo,
Vit\'oria, ES, 29075-910, Brazil}
\email{renan.mezabarba@ufes.br}

\author[L. F. Aurichi]{Leandro F. Aurichi}
\address{Instituto de Ci\^encias Matem\'aticas e de Computa\c c\~ao,
Universidade de S\~ao Paulo,
S\~ao Carlos, SP, 13560-970, Brazil}
\email{aurichi@icmc.usp.br}

\author[L. R. Junqueira]{Lúcia R. Junqueira}
\address{Instituto de Matem\'atica e Estat\'istica,
Universidade de S\~ao Paulo,
S\~ao Paulo, SP, 05508-900, Brazil}
\email{lucia@ime.usp.br}


\begin{abstract}
We investigate the notion of productive cellularity of arbitrary posets and topological spaces. Particularly, by working with families of antichains ordered with reverse inclusion, we give necessary and sufficient conditions to determine whether a poset or a topological space is productively ccc.
\end{abstract}

\maketitle

\section*{Introduction}

A topological space $X$ satisfies the \emph{countable chain condition} (\emph{ccc} for short) if there is no uncountable family of nonempty pairwise disjoint open sets. Its \emph{cellularity}, denoted $c(X)$, is the natural generalization: the supremum of all cardinals of the form $|A|$, where $A$ stands for a family of nonempty pairwise disjoint open sets of $X$.

Separable spaces are the easiest examples of ccc spaces, with the converse holding for metrizable spaces. In the same way, the \emph{density} of $X$, which is the least cardinality of a dense subset of $X$ and is denoted $d(X)$, is such that $c(X)\leq d(X)$, with the reverse inequality holding for metrizable spaces. Since the product of separable spaces is again separable, one would expect a similar behavior with cellularity, but the situation is quite delicate.

Indeed, in the realm of Martin's Axiom (MA) plus the negation of the Continuum Hypothesis ($\neg$CH), one can prove that every product of ccc spaces is a ccc space\footnote{Folklore, but possibly due to Juhász.}. On the other hand, a Suslin line turns out to be a ccc space whose square is not ccc\footnote{Attributed to Kurepa.}. Since each of the statements ``MA + $\neg$CH''and ``there exists a Suslin line'' are independent of ZFC\footnote{The reader may find the details in Kunen~\cite{Kunen}.}, it follows that productivity of ccc spaces is itself independent of ZFC.

Despite of this odd behavior concerning foundations, there are well known classes of \emph{productively ccc} spaces which turns out to be useful. For instance, arbitrary products of separable spaces are ccc, a result due to Fremlin and discussed again in Section 4. This implies that $\RR^X$ is ccc for every space $X$ and, since $C_p(X)$ is dense in $\RR^X$, it follows that $C_p(X)$ is ccc for every topological space $X$. Todor\v{c}evi\'{c}'s paper~\cite{stevo.chaintopology} discuss other applications of productiveness of the countable chain condition in topology.

However, the previous observations regarding foundations do not settle a related but different problem: which properties must a space $X$ have in order to guarantee that $X\times Y$ is a ccc space whenever $Y$ is a ccc space? Of course, these properties should refer only to $X$ or to spaces directly related to it, in order to avoid trivial answers. This is precisely what we do in this work, but in the slightly more general context of preordered sets (posets for short).

The paper is organized as follows. In the first section we recall some definitions about posets and fix notations. The second section contains our main result, the characterization of the productive cellularity of a poset. In the third section we investigate cardinal invariants related to productive cellularity, while the fourth section deals with the generalization of some classic theorems about ccc topological spaces to the present context. Finally, the last section is dedicated to discussing new perspectives for old (and open) problems concerning \emph{productively ccc posets}. Along the text, $\kappa$ and $\lambda$ denote infinite cardinals.

\section{Posets and their spectra}

We say that $(\PP,\leq)$ is a \emph{poset} if the set $\PP$ is endowed with a \emph{preorder} $\leq$, which is a reflexive and transitive binary relation, and it is called a \emph{partial order} if in addition $\leq$ is antisymmetric. Two elements $p$ and $q$ of a poset $(\PP,\leq)$ are said to be \emph{compatible} if there exists an $r\in\PP$ such that $r\leq p,q$. Naturally, we say that $p$ and $q$ are \emph{incompatible} if they are not compatible, what we abbreviate with $p\,\bot\,q$. Finally, a subset $A\subset \PP$ of pairwise incompatible elements is called an \emph{antichain} of $\PP$.

The \emph{cellularity} of a poset $\PP$, denoted by $c(\PP)$, is the supremum of all infinite cardinals of the form $|A|$ for some antichain $A\subset \PP$. In this way, the cellularity of a topological space $X$ is precisely the cellularity of the poset $\OO_X$, where $\OO_X$ denotes the family of nonempty open sets of $X$ partially ordered by inclusion.

For posets $(\PP,\leq)$ and $(\QQ,\preceq)$, we can consider the set $\PP\times\QQ$ preordered by the relation $\sqsubseteq$, which is defined by the rule 
\[(p,q)\sqsubseteq (p',q')\Leftrightarrow p\leq p'\,\,and\,\, q\preceq q'.\]

Then, it is not hard to see that for topological spaces $X$ and $Y$, the cellularity of $X\times Y$ is the cellularity of the poset $\OO_X\times\OO_Y$ endowed with the product order. However, even more is true. Let us define the \emph{cellular spectrum}\footnote{The word ``spectrum'' is a reference to the ``frequency spectrum'', considered by Arhangel'skii in \cite{arhantight}, a class of cardinals related to the tightness of products of topological spaces.} of a poset $\PP$, denoted $\spec{\PP}$, as the class

\[\spec{\PP}:=\left\{\kappa\geq \aleph_0:\forall \QQ\,(c(\QQ)\leq \kappa\Rightarrow c(\PP\times\QQ)\leq \kappa)\right\}.\]

Notice that a poset $\PP$ is \emph{productively ccc} precisely when $\aleph_0\in\spec{\PP}$. Also, defining the cellular spectrum $\spec{X}$ of a topological space in the same way does not give a new class of cardinals, since they are equal, as expected.

\begin{proposition}\label{spectop=specposet}
For a topological space $X$ one has $\spec{X}=\spec{\mathbb{O}_X}$.
\end{proposition}

\begin{proof}
Suppose $\kappa\in\spec{X}$ and let $\QQ$ be a poset with $c(\QQ)\leq\kappa$. Denote by $Y$ the set $\QQ$ endowed with the topology generated by the sets of the form $\{s\in\QQ:s\leq q\}$, which clearly satisfies $c(Y)=c(\QQ)$. Now, the hypothesis implies $c(X\times Y)\leq \kappa$, while a straightforward calculation yields \[c(X\times Y)=c(\mathbb{O}_X\times\mathbb{O}_Y)=c(\mathbb{O}_X\times\QQ),\]
showing that $\kappa\in\spec{\mathbb{O}_X}$.

The converse is trivial.
\end{proof}

In the PhD thesis of the first author \cite{thesis}, it was given necessary and sufficient conditions for a poset $\PP$ to be productively ccc\footnote{It is worthwhile to mention that the idea is an adaption of the methods applied by Aurichi and Zdomskyy \cite{Aur-Zdom} to characterize productive Lindelöf spaces.}, i.e., for $\aleph_0\in\spec{\PP}$. The considerations in the next section are the natural generalizations of the countable case.

\section{The main theorem}

Our goal in this section is to determine necessary and sufficient conditions in order to decide whether a given cardinal $\kappa\geq \aleph_0$ belongs to $\spec{\PP}$. To this end, we say that a family $\mathscr{A}$ of antichains of $\mathbb{P}$ is a $\kappa$-\emph{large family} if $\left|\bigcup\mathscr{A}\right|\geq \kappa^{+}$, and we denote by $\mL_{\kappa}(\PP)$ the collection of all such families. 
Finally, for a $\kappa$-large family $\mathscr{A}\in\mL_\kappa(\PP)$, we set $\mathscr{F}(\mathscr{A})=\bigcup_{\mA\in\mathscr{A}}[\mA]^{<\aleph_0}$, partially ordered by the reverse inclusion relation.

We shall use posets of the form $\mathscr{F(A)}$ in order to characterize the spectrum of $\PP$. This can be done because incompatibility conditions in $\mathscr{F(A)}$ translate to compatibility conditions of $\PP$, as we show in the following lemma.

\begin{lemma}\label{tech1}
Let $\mathscr{A}\in\mL_\kappa(\PP)$ and let $P,Q\in\mathscr{F(A)}$. Then $P\,\bot\,Q$ in $\mathscr{F(A)}$ if, and only if, $P\cup Q\not\in\mathscr{F(A)}$.
\end{lemma}
\begin{proof}
Note that if $P\cup Q\in\mathscr{F(A)}$, then $P,Q\subset P\cup Q$, showing that $P\,\bot\,Q$ does not hold. Conversely, if some $R$ in $\mathscr{F(A)}$ contains both $P$ and $Q$, then there is an $\mA\in\mathscr{A}$ such that $R\subset \mA$, showing that $P\cup Q$ is a finite subset of $\mA$, i.e., $P\cup Q\in\mathscr{F(A)}$.
\end{proof}

\begin{theorem}\label{main}
Let $\PP$ be a poset. Then $\kappa\in\spec{\PP}$ if and only if $c(\mathscr{F}(\mathscr{A}))>\kappa$ for all $\mathscr{A}\in\mL_{\kappa}(\PP)$.
\end{theorem}
\begin{proof}
If $\kappa\in\spec{\PP}$ and $\mathscr{A}\in\mL_\kappa(\PP)$ is such that $c(\mathscr{F}(\mathscr{A}))\leq \kappa$, then we have that $c(\PP\times\mathscr{F}(\mathscr{A}))\leq \kappa$. Now, let 
\[T:=\left\{(p,\{p\}):p\in\bigcup\mathscr{A}\right\}.\]

Since $\mathscr{A}$ is $\kappa$-large, the family $T$ cannot be an antichain in $\PP\times\mathscr{F}(\mathscr{A})$. Thus, there are $p,p'\in \bigcup\mathscr{A}$ with $p\ne p'$, $r\in\PP$ and $F\in\mathscr{F}(\mathscr{A})$ such that
\[(r,F)\sqsubseteq (p,\{p\}),(p',\{p'\})\]
implying  $p\not\!\!\bot p'$ and $\{p,p'\}\subseteq F\subseteq \mA$ for some $\mA\in\mathscr{A}$, showing that $\mA$ is not an antichain, a contradiction.

Conversely, supposing $\kappa\not\in\spec{\PP}$, we shall obtain a $\kappa$-large family $\mathscr{A}$ such that $c(\mathscr{F}(\mathscr{A}))\leq \kappa$. Let $\QQ$ be a poset witnessing $\kappa\not\in \spec{\PP}$, i.e., with $c(\QQ)\leq \kappa$ and such that there exists an antichain $\mW\subset \PP\times\QQ$ with $|\mW|=\kappa^+$. For each $r\in\QQ$, let $\mA_r:=\{p\in\PP:\exists q\in\QQ (r\leq q\text{ and }(p,q)\in\mW)\}.$
We claim that $\mathscr{A}:=\{\mA_r:r\in\QQ\}$ is the desired $\kappa$-large family.

Note that $\mA_r$ is clearly an antichain for each $r\in\QQ$, while $|\bigcup\mathscr{A}|=\kappa^+$ holds by the pigeonhole principle,
showing that $\mathscr{A}\in\mL_\kappa(\PP)$. It remains to show that $c(\mathscr{F}(\mathscr{A}))\leq \kappa$.
Indeed, for if $\mF\subset\mathscr{F}(\mathscr{A})$ is such that $|\mF|=\kappa^+$, for each $F\in\mF$ we take $r_F\in\QQ$ with $F\subset \mA_{r_F}$. Now we consider the set $\mR:=\{r_F:F\in\mF\}\subset\QQ$, that we shall use to obtain the desired inequality.

There are two cases:
\begin{enumerate}
\item if $|\mR|\leq \kappa$, then the pigeonhole principle gives $F,G\in\mF$ with $F\ne G$ and $r\in\mR$ such that $F\cup G\subset \mA_r$, showing that $F\cup G\in\mathscr{F(A)}$;
\item if $|\mR|=\kappa^+$, then $c(\QQ)\leq \kappa$ gives $F,G\in\mF$ with $F\ne G$ and $r\in\QQ$ such that $r\leq r_F,r_G$, showing that $\mA_{r_F}\cup\mA_{r_G}\subset \mA_r$, from which it follows that $F\cup G\in\mathscr{F(A)}$.
\end{enumerate}
In both cases, we obtain $F,G\in\mF$ with $F\ne G$, such that $F\cup G\in\mathscr{F}(\mathscr{A})$, which is equivalent to say that $F\,\bot\,G$ by the previous lemma, showing that $\mF$ is not an antichain of $\mathscr{F(A)}$, as desired.
\end{proof}

\begin{corollary}\label{maincor}
A poset $\PP$ is productively ccc if, and only if, $\mathscr{F(A)}$ is not ccc for all $\mathscr{A}\in\mL_\kappa(\PP)$.
\end{corollary}

The above characterizations become clearer in their contrapositive versions. For instance, Corollary~\ref{maincor} says that if a poset $\PP$ is not productively ccc, then there is a witness of the form $\mathscr{F(A)}$ for some $\mathscr{A}\in\mL_{\aleph_0}(\PP)$. Thus, if we have $\mL_{\aleph_0}(\PP)=\emptyset$, then $\PP$ is \emph{vacuously} productively ccc, since there are no witnesses to the contrary. This gives a very clean proof for the well known fact that countable posets are productively ccc. More generally, we have the following.
\begin{corollary}\label{trivial.cor}
If $|\PP|<\kappa$, then $\kappa\in\spec{\PP}$.
\end{corollary}

%

\section{A few inhabitants of the spectrum}

Besides showing that $\spec{\PP}$ is nonempty for all posets $\PP$, Corollary~\ref{trivial.cor} indicates that the possibly interesting cardinals in the cellular spectrum are smaller or equal to $|\PP|$. In particular, it makes sense to define the \emph{productive cellularity} of $\PP$ to be the cardinal 
\[\operatorname{pc}(\PP):=\min\spec{\PP},\]
in reference to the fact that $\PP$ is productively ccc if, and only if, $\operatorname{pc}(\PP)=\aleph_0$.

Since any poset $\TT$ with a single element satisfies $c(\TT)\leq \kappa$ for all $\kappa\geq \aleph_0$, it follows that for every poset $\PP$ one has $c(\PP)\leq \operatorname{pc}(\PP)$, from which it follows
\begin{equation}
\label{ineq1}
c(\PP)\leq \operatorname{pc}(\PP)\leq |\PP|^{+}.\end{equation}
We shall explore the gap between the cardinals $\operatorname{pc}(\PP)$ and $|\PP|^{+}$ through the rest of this section.

Recall that a subset $D$ of a poset $\PP$ is called \emph{dense} if for all $p\in \PP$ there is a $d\in D$ such that $d\leq p$. The \emph{density} of $\PP$, denoted by $d(\PP)$, is the least infinite cardinal of the form $|D|$ with $D\subset \PP$ dense, which is a generalization of the separability in topological spaces.

Since the cardinality of an antichain of $\PP$ is bounded by the cardinality of every dense subset of $\PP$, it follows immediately that $c(\PP)\leq d(\PP)$. This inequality can strengthen in the following way.

\begin{theorem}\label{denspec}
If $\PP$ is a poset, then $d(\PP)\in\spec{\PP}$.
\end{theorem}

\begin{proof}
Let $D\subset\PP$ be a dense subset and call $\kappa:=|D|$. We shall prove that $\kappa$ belongs to the cellular spectrum of $\PP$. By Theorem~\ref{main}, we need to take $\mathscr{A}\in\mL_{\kappa}(\PP)$ and show that $c(\mathscr{F}(\mathscr{A}))>\kappa$. Since $D$ is dense, it follows that for each $a\in\bigcup\mathscr{A}$ there exists a $\delta(a)\in D$ such that $\delta(a)\leq a$. Hence there exists a $d\in D$ such that the set $A:=\left\{a\in\bigcup\mathscr{A}:\delta(a)=d\right\}$ has cardinality at least $\kappa^+$. Finally, since $d\leq a$ for all $a\in A$, one can readily sees that the family $\{\{a\}:a\in A\}$ witnesses the inequality $c(\mathscr{F}(\mathscr{A}))>\kappa$, as desired.
\end{proof}

The arguments used above actually improve Corollary~\ref{trivial.cor}, allowing one to prove the following.
\begin{corollary}\label{trivial.den}
If $\kappa\geq d(\PP)$, then $\kappa\in\spec{\PP}$.
\end{corollary}
 In particular, \eqref{ineq1} can be replaced by
\begin{equation}
\label{ineq2}c(\PP)\leq \operatorname{pc}(\PP)\leq d(\PP).
\end{equation}

Note that in order to finish the proof of Theorem~\ref{denspec}, we used a very strong property of the family $A$, namely the existence of $d\in \PP$ such that $d\leq a$ for all $a\in A$. As we shall see below, this condition can be relaxed.

For a natural number $n\geq 2$, a subset $A\subset \PP$ is called $n$-\emph{linked} if for all $F\in [A]^{n}$ there exists $p_A\in\PP$ such that $p_A\leq p$ for each $p\in F$; $A$ is called \emph{centered} if $A$ is $n$-linked for all $n\geq 2$.  Then we have the following.
\begin{lemma}\label{tech2}
Let $\PP$ be a poset and $\mathscr{A}\in\mL_{\kappa}(\PP)$. If a subset $A\subset \bigcup\mathscr{A}$ is $n$-linked for some natural number $n\geq 2$, then the family $\{\{a\}:a\in A\}$ is an antichain in $\mathscr{F(A)}$.
\end{lemma}

\begin{proof}
For $a,b\in A$ with $a\ne b$, there is an $r\in \PP$ such that $r\leq a,b$. Then we have $\{a\},\{b\}\in \mathscr{F(A)}$ while $\{a\}\cup\{b\}=\{a,b\}\not\in\mathscr{F(A)}$, yielding $\{a\}\,\bot\,\{b\}$, by Lemma~\ref{tech1}.
\end{proof}

Although the above lemma seems to be innocuous, it has some interesting consequences. Let us recall a few more concepts in order to apply Lemma~\ref{tech2}.
Following~\cite{Stevo87}, we say that a poset $\PP$ has the \emph{K$_n$-property} if for each $A\in[\PP]^{\aleph_1}$ there exists an $n$-linked subset $B\in[A]^{\aleph_1}$. By replacing the occurrence of the term ``$n$-linked'' with ``centered'', we obtain the property usually called \emph{$\aleph_1$-precaliber}, but for sake of brevity we shall refer to it simply by \emph{$K_\omega$-property}. The letter ``K'' is a reference to Knaster, who first considered this type of property, for $n=2$.

In the same way cellularity generalizes the countable chain condition, we define below the \emph{Knaster invariants of $\PP$} in order to generalize $K_n$ and $K_\sigma$ properties. More precisely, for each natural number $n\geq 2$ we define the cardinal
\begin{equation}\mathscr{K}_n(\PP):=\min\{\kappa\geq \aleph_0:\forall A\in[\PP]^{\kappa^{+}}\exists B\in[A]^{\kappa^{+}}\left(B\text{ is $n$-linked}\right)\},\end{equation}
and we let
\begin{equation}\mathscr{K}_\omega(\PP):=\min\{\kappa\geq \aleph_0:\forall A\in [\PP]^{\kappa^{+}}\exists B\in [A]^{\kappa^{+}}\left(B\text{ is centered}\right)\}.\end{equation}

Note that for a poset $\PP$ and an ordinal $\alpha\in [2,\omega]$, $\PP$ has the $K_\alpha$-property if, and only if, $\mathscr{K}_\alpha(\PP)=\aleph_0$. The relations between the Knaster properties with the countable chain condition are in some sense preserved in the spectral context.

For a dense subset $D\subset \PP$ with $|D|=\kappa$, the same reasoning applied in Theorem~\ref{denspec} allows one to prove that for every $A\in [\PP]^{\kappa^{+}}$ there is a centered subset $B\in [A]^{\kappa^{+}}$, showing that $\mathscr{K}_\omega(\PP)\leq d(\PP)$. Since we clearly have $\mathscr{K}_\alpha(\PP)\leq\mathscr{K}_{\beta}(\PP)$ for $\alpha\leq \beta\leq \omega$, it follows that
\[\mathscr{K}_2(\PP)\leq \mathscr{K}_n(\PP)\leq\mathscr{K}_{n+1}(\PP)\leq \mathscr{K}_\omega(\PP)\leq d(\PP)\]
holds for every poset $\PP$. We now put $\operatorname{pc}(\PP)$ in the above inequalities.

\begin{theorem}
If $\PP$ is a poset, then $\mathscr{K}_2(\PP)\in\spec{\PP}$.
\end{theorem}
\begin{proof}
Let $\kappa:=\mathscr{K}_2(\PP)$ and let $\mathscr{A}\in\mL_\kappa(\PP)$. Since $|\bigcup\mathscr{A}|>\kappa$, there exists an $A\subseteq \bigcup\mathscr{A}$ such that $|A|=\kappa^{+}$. Now, there exists a $2$-linked subset $B\in[A]^{\kappa^{+}}$, so the conclusion follows from Lemma~\ref{tech2}.
\end{proof}

Differently of what happened in Theorem~\ref{denspec}, we are not able to adapt the previous argument to show that every $\kappa\geq \mathscr{K}_2(\PP)$ belongs to $\spec{\PP}$. Still, Lemma~\ref{tech2} can be used similarly to prove that the cardinal invariants $\mathscr{K}_\alpha(\PP)$ belongs to $\spec{\PP}$ for all $\alpha\in[2,\omega]$.
In summary, for every poset $\PP$ and every natural number $n\geq 2$, we have
\begin{equation}
\label{ultimate.ineq}c(\PP)\leq \operatorname{pc}(\PP)\leq \mathscr{K}_2(\PP)\leq \mathscr{K}_n(\PP)\leq \mathscr{K}_{n+1}(\PP)\leq\mathscr{K}_\sigma(\PP)\leq d(\PP).\end{equation}

\section{The spectra of products and some topological translations}

Although the (productive) cellularity of posets may be interesting by itself, the topological interpretations of the previous results deserve some attention. For a warming up example, the topological counterpart of Theorem \ref{denspec} says that separable spaces are productively ccc. Indeed, similar to what occurs with cellularity, the \emph{density} of a topological space $X$, denoted by $d(X)$, is the density of the poset $\OO_X$. Thus, in this context, Theorems~\ref{spectop=specposet} and~\ref{denspec} together say that the density of a space $X$ belongs to $\spec{X}$ and, since $X$ is separable if and only if $d(X)=\aleph_0$, our claim follows. However, even more is known to be true:

\begin{proposition}
[Fremlin~\cite{Fremlinbook}, Corollary 12J] Every product of separable spaces is productively ccc.
\end{proposition}

The above proposition make us wonder about the behavior of the cellular spectrum of a poset with respect to products. 
The very definition of the spectrum implies that for posets $\PP$ and $\QQ$ one has
\[\kappa\in\spec{\PP}\cap\spec{\QQ}\Rightarrow\kappa\in\spec{\PP\times\QQ},\]
showing that $\spec{\PP}\cap \spec{\QQ}\subset\spec{\PP\times\QQ}$. The reverse inclusion follows from the next easy lemma, whose proof is left for the reader.

\begin{lemma}\label{proj}
If $\varphi\colon\PP\to\QQ$ is an increasing function from the poset $\PP$ onto the poset $\QQ$, then $\spec{\PP}\subset\spec{\QQ}$.
\end{lemma}

\begin{theorem}\label{thatone}
If $\PP$ and $\QQ$ are posets, then $\spec{\PP\times\QQ}=\spec{\PP}\cap\spec{\QQ}$.
\end{theorem}
\begin{proof}
We already have $\spec{\PP}\cap\spec{\QQ}\subset \spec{\PP\times\QQ}$. Now, since the projections $\PP\times\QQ\to \PP$ and $\PP\times\QQ\to\QQ$ are both increasing and surjective, the reverse inclusion follows from the previous lemma.\end{proof}

In order to extend this result for arbitrary products of posets, we need to consider a slightly different product, closer to the topological counterpart of arbitrary products. We follow the definitions presented by Kunen in~\cite{Kunen}, where the reader may find more details. 

Let $\{\PP_i:i\in I\}$ be a nonempty family of posets such that for each $i\in I$ there is a largest element $1_i\in\PP_i$. Such posets are called \emph{forcing posets} in \cite{Kunen}. The \emph{finite support product} of the forcing posets $\PP_i$, denoted by $\prod_{i\in I}^{\operatorname{fin}}\PP_i$, is the subset of $\prod_{i\in I} \PP_i$ whose elements are those $I$-tuples $f$ such that $|\{i\in I:f_i\ne 1_i\}|<\aleph_0$, endowed with the coordinate-wise preorder. In some sense, this is the order-theoretic version of the standard topology on arbitrary products of topological spaces.

\begin{theorem}
For a nonempty family $\{\PP_i:i\in I\}$ of forcing posets one has
$\spec{\prod_{i\in I}^{\operatorname{fin}}\PP_i}=\bigcap_{i\in I}\spec{\PP_i}$.
\end{theorem}
\begin{proof}
The inclusion $\spec{\prod_{i\in\mI}^{\operatorname{fin}}\PP_i}\subset \bigcap_{i\in\mI}\spec{\PP_i}$ follows from Lemma~\ref{proj}. On the other hand, the reverse inclusion can be proved with a straightforward application of the $\Delta$-system lemma in conjunction with Theorem~\ref{thatone}.
\end{proof}

Once the above theorem is established, Fremlin's result about separable spaces becomes the topological counterpart of the following.

\begin{corollary}
If $\{\PP_i:i\in I\}$ is a nonempty family of forcing posets, then $\sup_{i\in I}d(\PP_i)\in\spec{\prod_{i\in I}^{\operatorname{fin}}\PP_i}$.
\end{corollary}
\begin{proof}
Since $d(\PP_j)\leq \sup_{i\in I}d(\PP_i)$ and $d(\PP_j)\in \spec{\PP_j}$, it follows from Corollary~\ref{trivial.den} that $\sup_{i\in I}d(\PP_i)\in\spec{\PP_j}$ for all $j\in I$, showing that 
\[\sup_{i\in I}d(\PP_i)\in\bigcap_{i\in I}\spec{\PP_i}=\spec{\prod_{i\in I}^{\operatorname{fin}}\PP_i}.\qedhere\]
\end{proof}

\section{Further questions and comments}

Corollary~\ref{trivial.den} and the absence of a similar result for the Knaster invariants suggest a natural  question about the behavior of the cardinals in the cellular spectrum. More precisely:

\begin{question}
Let $\PP$ be a poset. Does every cardinal $\kappa$ such that $\operatorname{pc}(\PP)< \kappa< d(\PP)$ belongs to $\spec{\PP}$?
\end{question}

Concerning the Knaster invariants, we still do not know if they are consistently different from each other. On the other hand, they all coincide under a standard assumption.

\begin{example} Assuming the existence of a Suslin Line $R$, one has $c(R)=\aleph_0$, while $\operatorname{pc}(R)>\aleph_0$ since $R$ is not productively ccc. On the other hand, the inequality $d(X)\leq c(X)^+$ holds for every LOTS $X$. Thus, we have $d(R)=\aleph_1$, from which it follows that all the Knaster invariants of $R$ collapse to $\aleph_1$.
\end{example}

It may also be interesting to explore the connections of the cellular spectrum with Martin's Axiom related topics. As we mentioned earlier, the standard strategy to show that MA+$\neg$CH implies that \emph{every ccc poset is productively ccc} starts by showing that \emph{every ccc poset has the $K_\omega$-property}. Note that we can restate both assertions, respectively, as the following implications:
\begin{align}
\label{PCCC}\forall \PP\quad c(\PP)=\aleph_0&\Rightarrow \operatorname{pc}(\PP)=\aleph_0;\\
\label{MA1}\forall \PP \quad c(\PP)=\aleph_0&\Rightarrow \mathscr{K}_\omega(\PP)=\aleph_0.\end{align}

Now, since $\operatorname{pc}(\PP)\leq \mathscr{K}_\omega(\PP)$, it follows immediately that \eqref{MA1}$\,\Rightarrow\,$\eqref{PCCC}. Although it is not completely well known, \eqref{MA1} is equivalent to MA$_{\aleph_1}$, thanks to the next proposition, due to Todor\v{c}evi\'{c} and Veli\v{c}kovi\'{c}~\cite{Stevo87}.

\begin{proposition}
[Todor\v{c}evi\'{c} and Veli\v{c}kovi\'{c}~\cite{Stevo87}, Theorem 3.4] MA$_{\aleph_1}$ holds if, and only if, every uncountable ccc poset has an uncountable centered subset.
\end{proposition}

In~\cite{LT}, up to terminology, Larson and Todor\v{c}evi\'{c} asks whether any of the assumptions
\begin{enumerate}
\item $\forall \PP$ $c(\PP)=\aleph_0\Rightarrow \mathscr{K}_2(\PP)=\aleph_0$ or
\item $\forall \PP$ $c(\PP)=\aleph_0\Rightarrow \operatorname{pc}(\PP)=\aleph_0$
\end{enumerate}
imply MA$_{\aleph_1}$. Thus, after all we have done so far, Larson and Todor\v{c}evi\'{c}'s questions\footnote{And currently open.} suggest the following.

\begin{question}
Does MA$_{\aleph_1}$ implies $c(\PP)=\operatorname{pc}(\PP)$ for every poset $\PP$? Does the converse hold?
\end{question}
\begin{question}
Does MA$_{\aleph_1}$ implies $c(\PP)=\mathscr{K}_2(\PP)$ for every poset $\PP$? Does the converse hold?
\end{question}

\bibliography{ccc}{}
\bibliographystyle{abbrv}

\end{document}